\documentclass[11pt]{article}
\usepackage{amsmath,amssymb,amsthm,mathtools,geometry,bm,url,tikz,xcolor}
\usetikzlibrary{positioning,arrows}

\usepackage{booktabs}

\usepackage[table]{xcolor}
\definecolor{toprow}{gray}{0.82}
\definecolor{evenrow}{gray}{0.93}

\title{Asymptotically optimal approximate Hadamard matrices}
\author{Boris Alexeev \and John Jasper\thanks{Department of Mathematics and Statistics, Air Force Institute of Technology, 2950 Hobson Way, Wright-Patterson Air Force Base, OH 45433, USA} \and Dustin G.\ Mixon\thanks{Department of Mathematics, The Ohio State University, 231 West 18th Avenue, Columbus, OH 43210, USA} \thanks{Translational Data Analytics Institute, The Ohio State University, 1760 Neil Avenue, Columbus, OH 43210, USA}}
\date{}

\theoremstyle{plain}
\newtheorem{theorem}{Theorem}
\newtheorem{lemma}[theorem]{Lemma}
\newtheorem{proposition}[theorem]{Proposition}

\theoremstyle{definition}

\newtheorem{problem}[theorem]{Problem}

\begin{document}
\maketitle

\begin{abstract}
An approximate Hadamard matrix is a well-conditioned square matrix with all entries in~\(\{\pm1\}\). 
We measure the quality of a matrix by its condition number, i.e., the ratio of its largest and smallest singular values. 
We prove that for any fixed positive \(\alpha<17/92\), every sufficiently large dimension admits an approximate Hadamard matrix with condition number at most \(1+n^{-\alpha}\). 
In particular, the smallest possible condition number tends to \(1\) as \(n\to\infty\). 
Conversely, there exists an absolute constant \(c>0\) such that for every sufficiently large \(n\not\equiv0\pmod4\), every \(n\times n\) matrix with entries in \(\{\pm1\}\) has condition number at least \(1+c(\log n)/n\). 
Along the way, we resolve a problem of Jaming and Matolcsi concerning flat orthogonal matrices, and we conclude by describing several explicit infinite families of approximate Hadamard matrices.
\end{abstract}

\noindent\textbf{2020 Mathematics Subject Classification.}
Primary 15B34; Secondary 05B20, 15A12, 05D10.

\section{Introduction}

A \textit{Hadamard matrix} is a multiple of an orthogonal matrix with all entries in $\{\pm1\}$.
Hadamard matrices date back to Sylvester's 1867 constructions of orthogonal sign matrices~\cite{Sylvester:67} and Hadamard's 1893 work on maximal determinants~\cite{Hadamard:93}. 
In the 1930s, Paley~\cite{Paley:33} constructed two infinite families and conjectured that a Hadamard matrix exists in every order divisible by \(4\). 
This statement, now known as the \emph{Hadamard conjecture}, remains open.
In this paper, we resolve a quantitative relaxation of this conjecture.

Given a real matrix $A$, let $\kappa(A)\in[1,\infty]$ denote the \textit{condition number} of $A$: 
\[
\kappa(A)=\frac{\sigma_{\max}(A)}{\sigma_{\min}(A)},
\]
where $\sigma_{\min}(A)$ and $\sigma_{\max}(A)$ denote the smallest and largest singular values of $A$, respectively.
(If $\sigma_{\min}(A)=0$, we put $\kappa(A)=\infty$.)
For each positive integer $n$, let $\kappa(n)$ denote the smallest possible condition number of an $n\times n$ matrix with all entries in $\{\pm1\}$, that is,
\[
\kappa(n)
:=\min_{A\in\{\pm1\}^{n\times n}}\kappa(A).
\]
Observe that $\kappa(n)\geq1$, with equality precisely when there exists a Hadamard matrix of order $n$. Moreover, equality can only occur when \(n\) equals \(1\), \(2\), or a multiple of $4$, since Hadamard matrices can only exist in these sizes (as a consequence of the well-known Proposition~\ref{prop.hadamard necessary condition} below).

Recently, Dong and Rudelson~\cite{DongR:24} showed that the sequence $\kappa(n)$ is uniformly bounded, and Steinerberger~\cite{Steinerberger:24} later showed this still holds when restricting to circulant matrices.
Steinerberger's result follows by taking the coefficients of the flat Littlewood polynomials constructed by Balister, Bollob\'as, Morris, Sahasrabudhe, and Tiba~\cite{BalisterBMST:20} as the first row of a circulant matrix. The stronger problem of constructing \emph{ultraflat} Littlewood polynomials, whose coefficients belong to \(\{\pm1\}\) and whose modulus is \((1+o(1))\sqrt{n}\) uniformly on the complex unit circle, remains open and is analogous to the present problem, albeit in the more restrictive circulant setting.
Steinerberger has since promoted this problem of \textit{approximate Hadamard matrices}, for example, in~\cite{Manskova:25,Steinerberger:notes}.
(In fact, he personally introduced the third author to the problem over a drink at the 5th Biennial Meeting of the Pacific Northwest Section of SIAM.)
Here are the two main questions posed by Steinerberger:
\begin{itemize}
\item[1.] Does it hold that $\displaystyle\lim_{n\to\infty}\kappa(n)=1$?
\item[2.] Are there explicit infinite families of approximate Hadamard matrices?
\end{itemize}
In this paper, we resolve both questions in the affirmative, and we pose a few follow-up questions.

In the next section, we establish that there exists $\alpha>0$ such that $ \kappa(n)\leq1+\frac{1}{n^\alpha}$ for all sufficiently large $n$.
As part of our proof, we resolve a $10$-year-old problem of Jaming and Matolcsi~\cite{JamingM:15} concerning flat orthogonal matrices.
Next, Section~\ref{sec.lower bound} uses ideas from Ramsey theory to obtain a lower bound of the form $\kappa(n)\geq 1+\frac{c\log n}{n}$ for all sufficiently large $n\not\equiv 0\bmod 4$.
In Section~\ref{sec.explicit families}, we identify a few infinite families of approximate Hadamard matrices.

\section{Upper bound}

In this section, we construct approximate Hadamard matrices of order $n$ with condition number approaching $1$ as $n\to\infty$.
Our approach follows the schematic
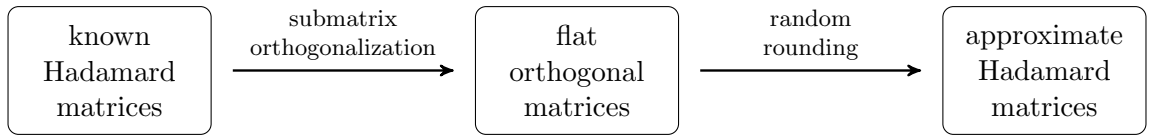
\begin{center}
\begin{tikzpicture}[
  node distance=4cm,
  box/.style={
    rectangle,
    draw,
    rounded corners,
    minimum width=2.7cm,
    minimum height=1.7cm,
    align=center
  },
  arrow/.style={
    ->,
    shorten >=8pt,
    shorten <=8pt,
    line width=0.9pt,
    >=stealth'
  }
]
\node[box] (craigen) {known\\Hadamard\\matrices};
\node[box, right=3.5cm of craigen] (flat) {flat\\orthogonal\\matrices};
\node[box, right=3.5cm of flat] (approx) {approximate\\Hadamard\\matrices};
\draw[arrow] (craigen) -- node[above, align=center, font=\footnotesize]{submatrix\\orthogonalization} (flat);
\draw[arrow] (flat) -- node[above, align=center, font=\footnotesize]{random\\rounding} (approx);
\end{tikzpicture}
\end{center}
and the result is the following:

\begin{theorem}
\label{thm.main result}
There exists $\alpha>0$ such that $\displaystyle \kappa(n)\leq1+\frac{1}{n^\alpha}$ for all sufficiently large $n$.
\end{theorem}

As we will see, with known constructions of Hadamard matrices, we can take $\alpha=\frac{17}{92}-\delta$ for any small $\delta>0$, while conditioned on the Hadamard conjecture, we can take $\alpha=\frac{1}{4}-\delta$.

First, we explain the second arrow of the above schematic.
In particular, we identify the related problem of \textit{flat orthogonal matrices}, introduced by Jaming and Matolcsi in~\cite{JamingM:15}.
In particular, for each positive integer $n$, let $u(n)$ denote the smallest possible entry-wise infinity norm of an $n\times n$ orthogonal matrix:
\[
u(n)
:=\min_{M\in O(n)}\|M\|_{\max}.
\]
Observe that $u(n)\geq 1/\sqrt{n}$, with equality precisely when there exists a Hadamard matrix of order~$n$.
It turns out that $u(n)$ approaches this lower bound when $n$ gets large, thereby resolving Problem~1.2 of~\cite{JamingM:15} in the affirmative:

\begin{theorem}
\label{thm.flat orthogonal matrices}
There exists $\epsilon\in(0,\frac{1}{2})$ such that $\displaystyle u(n)\leq\frac{1}{\sqrt{n}-n^\epsilon}$ for all sufficiently large $n$.
\end{theorem}

As we will see, with known constructions of Hadamard matrices, we can take $\epsilon=\frac{3}{23}+\delta$ for any small $\delta>0$, while conditioned on the Hadamard conjecture, we can take $\epsilon>0$ arbitrarily small.

Before proving this theorem, we first show that it implies Theorem~\ref{thm.main result}.
(As suggested by our schematic, the main idea is that randomly rounding a flat orthogonal matrix results in an approximate Hadamard matrix.)

\begin{proof}[Proof of Theorem~\ref{thm.main result}]
Select $M\in O(n)$ with $\|M\|_{\max}=u(n)$, and let $X$ denote the random matrix with independent entries in $\{\pm1\}$ such that $\mathbb{E}X=M/\|M\|_{\max}$.
Put $E:=X-\mathbb{E}X$.
We will use the matrix Bernstein inequality to show that 
\begin{equation}
\label{eq.matrix bernstein}
\mathbb{E}\|E\|_{\operatorname{op}}
\leq\sqrt{2(n-\tfrac{1}{u(n)^2})\log(2n)}+\tfrac{2}{3}\log(2n)
=:e(n).
\end{equation}
For now, we assume \eqref{eq.matrix bernstein} holds.
Fix a sample point $\omega$ with $\|E(\omega)\|_{\operatorname{op}}\leq\mathbb{E}\|E\|_{\operatorname{op}}$.
Then Weyl's inequality gives
\begin{align*}
\sigma_{\min}(X(\omega))
&\geq\sigma_{\min}(\mathbb{E}X)-\|E(\omega)\|_{\operatorname{op}}
\geq\frac{1}{u(n)}-e(n)
,\\
\sigma_{\max}(X(\omega))
&\leq\sigma_{\max}(\mathbb{E}X)+\|E(\omega)\|_{\operatorname{op}}
\leq\frac{1}{u(n)}+e(n),
\end{align*}
and so
\[
\frac{\sigma_{\max}(X(\omega))}{\sigma_{\min}(X(\omega))}
\leq\frac{1+u(n)e(n)}{1-u(n)e(n)}
=1+2u(n)e(n)+o\big(u(n)e(n)\big).
\]
Meanwhile, Theorem~\ref{thm.flat orthogonal matrices} implies
\[
u(n)e(n)
\leq\frac{\sqrt{4n^{1/2+\epsilon}
\log(2n)}+\tfrac{2}{3}\log(2n)}{\sqrt{n}-n^\epsilon}
=\big(1+o(1)\big)\cdot\frac{\sqrt{4\log n}}{n^{1/4-\epsilon/2}}.
\]
The result then follows by taking $\alpha<1/4-\epsilon/2$.

It remains to establish \eqref{eq.matrix bernstein}.
We shall follow Theorem~6.1.1 in Tropp's \textit{An Introduction to Matrix Concentration Inequalities}~\cite{Tropp:15}, which in turn requires us to identify a boundedness parameter $L$ and a variance parameter $v$ in our use case.
Let $E_{ij}$ denote the $(i,j)$ entry of the random matrix~$E$, let $e_i$ denote the $i$th standard basis element, and put $S_{ij}:=E_{ij}e_ie_j^\top$ so that $E=\sum_{i,j\in[n]}S_{ij}$.
Since $|E_{ij}|\leq 2$ almost surely, we have $\|S_{ij}\|_{\operatorname{op}}\leq 2=:L$ almost surely.
Next,
\[
\sum_{i,j\in[n]}\mathbb{E}[S_{ij}S_{ij}^\top]
=\sum_{i,j\in[n]}\operatorname{Var}(X_{ij})\cdot e_ie_i^\top
=\sum_{i,j\in[n]}\big(1-(\mathbb{E}X_{ij})^2\big)\cdot e_ie_i^\top
=(n-\tfrac{1}{u(n)^2}) \cdot I_n,
\]
and similarly, $\sum_{i,j\in[n]}\mathbb{E}[S_{ij}^\top S_{ij}]=(n-\tfrac{1}{u(n)^2}) \cdot I_n$.
It follows that
\[
v
:=\max\bigg\{\bigg\|\sum_{i,j\in[n]}\mathbb{E}[S_{ij}S_{ij}^\top]\bigg\|_{\operatorname{op}},\bigg\|\sum_{i,j\in[n]}\mathbb{E}[S_{ij}^\top S_{ij}]\bigg\|_{\operatorname{op}}\bigg\}
=n-\frac{1}{u(n)^2}.
\]
Thus, taking $d_1:=n$ and $d_2:=n$, Theorem~6.1.1 in~\cite{Tropp:15} gives
\begin{align*}
\mathbb{E}\|E\|_{\operatorname{op}}
\leq \sqrt{2v\log(d_1 + d_2)}+\tfrac{1}{3} L \log(d_1 +d_2)
=\sqrt{2(n-\tfrac{1}{u(n)^2})\log(2n)}+\tfrac{2}{3}\log(2n)
\end{align*}
i.e., \eqref{eq.matrix bernstein} holds.
\end{proof}

It remains to prove Theorem~\ref{thm.flat orthogonal matrices}.
Here's our approach:
Given a positive integer $n$, find a Hadamard matrix $H$ of order $m\geq n$ (with $m$ not much larger than $n$), and then ``orthogonalize'' an $n\times n$ submatrix of $H$.
The following lemma explains this latter step:

\begin{lemma}[submatrix orthogonalization]
\label{lem.submatrix orthogonalization}
Given an orthogonal matrix $\left[\begin{smallmatrix} A & B \\ C & D \end{smallmatrix}\right]$ with $I-A$ invertible, it holds that $D+C(I-A)^{-1}B$ is also orthogonal.
\end{lemma}

\begin{proof}
We will show that $M:=D+C(I-A)^{-1}B$ satisfies $\|Mu\|=\|u\|$ for every $u$.
Fix $u$.
Since $I-A$ is invertible, there is a unique $x$ such that $x = Ax + Bu$, namely, $x = (I-A)^{-1}Bu$.
Then
\[
\left[\begin{array}{cc} A & B \\ C & D \end{array}\right]\left[\begin{array}{c} x \\ u \end{array}\right]
=\left[\begin{array}{c} Ax+Bu \\ Cx+Du \end{array}\right]
=\left[\begin{array}{c} x \\ C(I-A)^{-1}Bu+Du \end{array}\right]
=\left[\begin{array}{c} x \\ Mu \end{array}\right].
\]
Since $\left[\begin{smallmatrix} A & B \\ C & D \end{smallmatrix}\right]$ is orthogonal, the result follows by taking the norm of both sides.
\end{proof}

Conveniently, running submatrix orthogonalization on a Hadamard matrix results in a \textit{flat} orthogonal matrix, provided $A$ has relatively few rows and columns:

\begin{lemma}
\label{lem.submatrix is flat orthogonal}
Given an $m\times m$ Hadamard matrix $H$, select $k<\sqrt{m}$ and consider the block decomposition $\frac{1}{\sqrt{m}}H = \left[\begin{smallmatrix} A & B \\ C & D \end{smallmatrix}\right]$ with $A\in\mathbb{R}^{k\times k}$.
Then $I-A$ is invertible and $\|D+C(I-A)^{-1}B\|_{\max}\leq\frac{1}{\sqrt{m}-k}$.
\end{lemma}

\begin{proof}
First, $\|A\|_{\operatorname{op}}^2\leq\|A\|_F^2=k^2/m<1$.
As such, $I-A$ is invertible and the Neumann series
\[
(I-A)^{-1} 
= \sum_{n=0}^\infty A^n
\]
converges.
Put $M:=D+C(I-A)^{-1}B$.
The size of the $(i,j)$ entry of $M$ is
\[
| D_{ij}+e_i^\top C(I-A)^{-1}Be_j |
\leq | D_{ij} | + \|e_i^\top C\|_2 \|(I-A)^{-1}\|_{\operatorname{op}} \|Be_j\|_2
=\frac{1}{\sqrt{m}}+\frac{k}{m}\cdot\|(I-A)^{-1}\|_{\operatorname{op}}.
\]
Next, we apply the triangle inequality over the Neumann series:
\[
\|(I-A)^{-1}\|_{\operatorname{op}}
\leq\sum_{n=0}^\infty \|A\|_{\operatorname{op}}^n
\leq\sum_{n=0}^\infty \left(\frac{k}{\sqrt{m}}\right)^n
=\frac{1}{1-k/\sqrt{m}}.
\]
All together, we have
\[
\|M\|_{\max}
\leq\frac{1}{\sqrt{m}}+\frac{k}{m}\cdot\frac{1}{1-k/\sqrt{m}}
=\frac{1}{\sqrt{m}-k}.
\qedhere
\]
\end{proof}

Finally, the plethora of known Hadamard matrices allows us to select $m$ close to $n$:

\begin{lemma}
\label{lem.hadamard gaps}
Fix positive integers $a$, $b$, and $c$ with the property that for every positive integer~$s$, there exists a Hadamard matrix of order $2^t(2s+1)$, where $t=t(s)=a\lfloor\frac{1}{b}\log_2 s\rfloor + c$.
Then for every sufficiently large integer $n$, there exists a Hadamard matrix of order m such that
\[
n
\leq m
< n + 2\cdot 2^{\frac{b(a+c+1)}{a+b}}\cdot n^{\frac{a}{a+b}}.
\]
\end{lemma}

Notably, the hypothesis of Lemma~\ref{lem.hadamard gaps} was proven by Craigen~\cite{Craigen:95} in 1995 with $a=4$, $b=6$, and $c=6$, resulting in an exponent of $\frac{a}{a+b}=\frac{2}{5}<\frac{1}{2}$.
More recently, Du and Jiang~\cite{DuJ:24} obtained $a=6$, $b=40$, and $c=16$, resulting in a smaller exponent of $\frac{a}{a+b}=\frac{3}{23}$.
(See~\cite{DuJ:24} for a discussion of the history of such results.)

\begin{proof}[Proof of Lemma~\ref{lem.hadamard gaps}]
By iteratively taking Kronecker products with a Hadamard matrix of order $2$, we have Hadamard matrices of order $2^t(2s+1)$ for every $t\geq t(s)$.
Notably, $t(s)$ is monotonically increasing in $s$.
Given an integer $t\geq c$, let $s(t)$ denote the largest integer $s$ for which $t\geq t(s)$.
Then we have Hadamard matrices of order $2^t(2s+1)$ for every $s\leq s(t)$.
This gives an arithmetic progression of Hadamard matrix orders with common difference $2^{t+1}$.
Given a positive integer $n$, find the smallest $t\geq c$ for which $n\leq 2^t(2s(t)+1)$.
Then the smallest $m\geq n$ in this arithmetic progression satisfies
\[
n\leq m < n + 2^{t+1}.
\]
It remains to bound $2^{t+1}$ in terms of $a$, $b$, and $c$.

First, we observe that $s(t)\geq 2^{b\lfloor \frac{t-c}{a}\rfloor}$ for $t\geq c$.
Indeed, this follows from the easily verified fact that $t(2^{b\lfloor \frac{t-c}{a}\rfloor})\leq t$.
Next, since $t$ was selected to be minimal, it either holds that $t=c$ or 
\[
n
>2^{t-1}(2s(t-1)+1)
\geq 2^{t}\cdot 2^{b\lfloor \frac{t-1-c}{a}\rfloor}
\geq 2^t\cdot 2^{b(\frac{t-1-c}{a}-1)}
=2^{\frac{t(a+b)}{a}}\cdot2^{-\frac{b(a+c+1)}{a}},
\]
which rearranges to $2^t<2^{\frac{b(a+c+1)}{a+b}}\cdot n^{\frac{a}{a+b}}$.
Since $n$ is sufficiently large, this latter bound is greater than $2^c$, and the result follows.
\end{proof}

\begin{proof}[Proof of Theorem~\ref{thm.flat orthogonal matrices}]
Put $a=6$, $b=40$, and $c=16$, and fix $\epsilon\in(\frac{a}{a+b},\frac{1}{2})$.
Given a sufficiently large integer $n$, then the main result in~\cite{DuJ:24} combined with Lemma~\ref{lem.hadamard gaps} gives a Hadamard matrix $H$ of order $m\in[n,n+n^\epsilon)$.
Put $k:=m-n$, and consider the block decomposition $\frac{1}{\sqrt{m}}H = \left[\begin{smallmatrix} A & B \\ C & D \end{smallmatrix}\right]$ with $A\in\mathbb{R}^{k\times k}$.
By Lemma~\ref{lem.submatrix is flat orthogonal}, $I-A$ is invertible, and $M:=D+C(I-A)^{-1}B$ satisfies $\|M\|_{\max}\leq\frac{1}{\sqrt{m}-k}$.
Furthermore, $M$ is orthogonal by Lemma~\ref{lem.submatrix orthogonalization}, and so
\[
u(n)
\leq \|M\|_{\max}
\leq\frac{1}{\sqrt{m}-k}
\leq\frac{1}{\sqrt{n}-n^\epsilon}.
\qedhere
\]
\end{proof}

\section{Lower bound}
\label{sec.lower bound}

In this section, we use ideas from Ramsey theory to obtain a lower bound on $\kappa(n)$:

\begin{theorem}
\label{thm.kappa lower bound}
There exists $c>0$ such that $\displaystyle\kappa(n)\geq 1+\frac{c\log n}{n}$ for all sufficiently large $n\not\equiv 0\bmod 4$.
\end{theorem}

Our approach is to identify a poorly conditioned principal submatrix of $A^\top A$ for any given $A\in\{\pm1\}^{n\times n}$.
Let $I_k$ and $J_k$ denote the $k\times k$ identity and all-ones matrices, respectively.
The following lemma indicates the sort of submatrix we target:

\begin{lemma}
\label{lem.off diag pos}
Fix positive integers $n\geq k>1$, and suppose $M\in\mathbb{R}^{k\times k}$ is positive semidefinite with $n$'s on the diagonal and positive integers on the off-diagonal.
Then $\kappa(M)\geq 1+\frac{k}{n-1}$, with equality precisely when $M=(n-1)I_k+J_k$.
\end{lemma}

\begin{proof}
In the case where $M=M_0:=(n-1)I_k+J_k$, it holds that $n-1$ is an eigenvalue with multiplicity $k-1$, and $n-1+k$ is an eigenvalue with multiplicity $1$.
As such, $\kappa(M)=1+\frac{k}{n-1}$.
Now suppose $M\neq M_0$.
Then
\[
\lambda_{\max}(M)
\geq\frac{\mathbf{1}_k^\top M\mathbf{1}_k}{\mathbf{1}_k^\top\mathbf{1}_k}
> \frac{kn+k(k-1)}{k}
=\lambda_{\max}(M_0),
\]
and so
\[
\lambda_{\min}(M)
\leq\frac{1}{k-1}\Big(\operatorname{tr}(M)-\lambda_{\max}(M)\Big)
<\frac{1}{k-1}\Big(\operatorname{tr}(M_0)-\lambda_{\max}(M_0)\Big)
=\lambda_{\min}(M_0).
\]
Thus, $\kappa(M)>\kappa(M_0)$.
\end{proof}

A similar proof gives the following:

\begin{lemma}
\label{lem.off diag neg}
Fix positive integers $n\geq k>1$, and suppose $M\in\mathbb{R}^{k\times k}$ is positive semidefinite with $n$'s on the diagonal and negative integers on the off-diagonal.
Then $\kappa(M)\geq 1+\frac{k}{n+1-k}$, with equality precisely when $M=(n+1)I_k-J_k$.
\end{lemma}

Finally, it will be useful to recall the following classic result (see Proposition~3 in \cite{BrowneEHC:21}), which in turn implies the famous necessary condition on Hadamard matrices.

\begin{proposition}
\label{prop.hadamard necessary condition}
There exist three orthogonal vectors in $\{\pm1\}^n$ if and only if $n$ is a multiple of $4$.
\end{proposition}

\begin{proof}[Proof of Theorem~\ref{thm.kappa lower bound}]
In the case where $n$ is odd, let $k$ denote the largest positive integer such that $n\geq R(k,k)$, where $R(k,k)$ denotes the $k$th diagonal Ramsey number; the classical Erd\H{o}s--Szekeres bound~\cite{ErdosS:35} gives
\[
R(k,k)
\leq \binom{2k-2}{k-1}<4^k,
\]
and so $k\geq\Omega(\log n)$.
Since $n$ is odd, the signs of the entries of $A^\top A$ determine a $2$-coloring of the edges of $K_n$.
(Here, the colors are ``positive'' and ``negative''.)
There is necessarily a monochromatic clique of size~$k$, so this determines a principal $k\times k$ submatrix $M$ of $A^\top A$ such that all off-diagonal entries have the same sign.
Thus, Lemmas~\ref{lem.off diag pos} and~\ref{lem.off diag neg} together give
\begin{equation}
\label{eq.lower bound on kappa}
\kappa(A)
=\sqrt{\kappa(A^\top A)}
\geq\sqrt{\kappa(M)}
\geq\sqrt{1+\frac{k}{n-1}}.
\end{equation}

In the case where $n$ is even, let $k$ denote the largest positive integer such that $n\geq R(3,k,k)$, where $R(3,k,k)$ denotes the appropriate multicolor Ramsey number; here, merging the last two colors and applying the Erd\H{o}s--Szekeres bound gives 
\[
R(3,k,k)
\leq R(3,R(k,k))
\leq \binom{R(k,k)+1}{2}
<16^k, 
\]
and so $k\geq\Omega(\log n)$.
Since $n$ is even, the signs of the entries of $A^\top A$ determine a $3$-coloring of the edges of $K_n$.
(Here, the colors are ``zero'', ``positive'', and ``negative''.)
There is either a triangle of color ``zero'' or a size-$k$ clique of color ``positive'' or ``negative''.
However, since $n$ is not a multiple of $4$, Proposition~\ref{prop.hadamard necessary condition} rules out the first possibility.
As such, there is a principal $k\times k$ submatrix $M$ of $A^\top A$ such that all off-diagonal entries have the same sign, and so the bound \eqref{eq.lower bound on kappa} holds.
\end{proof}

\section{Explicit approximate Hadamard matrices}
\label{sec.explicit families}

We say $C\in\mathbb{R}^{n\times n}$ is a \textit{conference matrix} if $C$ has $0$'s on the diagonal, $\pm1$'s on the off-diagonal, and satisfies $C^\top C=(n-1)I$.
Delsarte, Goethals, and Seidel~\cite{DelsarteGS:71} showed that, for every conference matrix, one may apply signed permutations to the rows and columns to obtain a conference matrix $C$ that is symmetric (when $n\equiv 2\bmod 4$) or skew-symmetric (when $n\equiv 0\bmod4$).
In the latter case, it holds that $C+I$ is a Hadamard matrix; in fact, all \textit{skew Hadamard} matrices arise in this way, and they are conjectured to exist for every $n\equiv0\bmod4$~\cite{Wallis:71}.
When $n\equiv2\bmod4$, it is natural to consider $\kappa(C+I)$.

\begin{lemma}
Given a symmetric conference matrix $C\in\mathbb{R}^{n\times n}$, it holds that $\kappa(C+I)=\frac{\sqrt{n-1}+1}{\sqrt{n-1}-1}$.
\end{lemma}

\begin{proof}
Since $C^2=C^\top C=(n-1)I$, the eigenvalues of $C$ reside in $\{\pm\sqrt{n-1}\}$.
Next, $\operatorname{tr}(C)=0$ implies that both numbers are eigenvalues with multiplicity $n/2$.
Then the eigenvalues of $C+I$ are $\sqrt{n-1}+1$ and $-\sqrt{n-1}+1$, and so the singular values are the absolute values of these: $\sqrt{n-1}+1$ and $\sqrt{n-1}-1$.
The quotient is the desired condition number.
\end{proof}

Notably, this implies $\kappa(n)=1+O(\frac{1}{\sqrt{n}})$ whenever there exists a symmetric conference matrix of order $n$.
This occurs if $n-1\equiv 1\bmod 4$ is a prime power, and only if $n-1$ is a sum of two perfect squares~\cite{Belevitch:50,BaloninS:online}.

Next, we take inspiration from yet another relaxation of the Hadamard matrix problem, namely, Hadamard's maximal determinant problem; see~\cite{BrowneEHC:21} for a recent survey.
One family of matrices $A\in\{\pm1\}^{n\times n}$ that emerges in this problem satisfies $A^\top A=(n-1)I+J$.
These are known as \textit{Barba matrices}.
Then $\lambda_{\max}(A^\top A)=2n-1$ and $\lambda_{\min}(A^\top A)=n-1$, and so 
\[
\kappa(A)
=\sqrt{\kappa(A^\top A)}
=\sqrt{\frac{2n-1}{n-1}}=\sqrt{2}+o(1).
\]
Such matrices exist, for example, whenever $n=2(q^2+q)+1$ for some odd prime power $q$; see Theorem~A in~\cite{NeubauerR:97}.
Notably, this implies $n\equiv 1\bmod 4$.

Another important family arises from \textit{supplementary difference sets (SDSs)}.
Specifically, when $n\equiv 2\bmod 4$, suppose there exist circulant matrices $R,S\in\{\pm1\}^{n/2\times n/2}$ that satisfy the autocorrelation identity $R^\top R+S^\top S=(n-2)I+2J$.
Then the matrix
\[
A=\left[\begin{array}{ll} R & \phantom{-}S \\ S^\top & -R^\top\!\!\end{array}\right]
\]
satisfies $A^\top A=I_2\otimes((n-2)I_{n/2}+2J_{n/2})$.
As such, $\lambda_{\max}(A^\top A)=2n-2$ and $\lambda_{\min}(A^\top A)=n-2$, and so
\[
\kappa(A)
=\sqrt{\kappa(A^\top A)}
=\sqrt{\frac{2n-2}{n-2}}=\sqrt{2}+o(1).
\]
Furthermore, such matrices exist whenever $n/2=q^2+q+1$ for some odd prime power $q$; see Theorem~1 in~\cite{Spence:75}.

See the appendix of the arXiv version of this paper for a table of smallest known condition numbers of approximate Hadamard matrices of order at most $30$.

\section{Discussion}
\label{sec.discussion}

In this paper, we showed that the smallest-possible condition number of $n\times n$ approximate Hadamard matrices approaches $1$ as $n\to\infty$, and we presented some explicit constructions.
Several open problems remain, and in this section, we highlight a few of them.
Our first problem asks for the rate at which the best condition number approaches $1$:

\begin{problem}
What is the largest $\alpha$ for which $\displaystyle\kappa(n)=1+\frac{f(n)}{n^{\alpha}}$ for some subpolynomial $f$?
\end{problem}

We currently know that $\frac{17}{92}\leq\alpha\leq1$.
Better upper bounds on gaps between Hadamard matrices will increase this lower bound, but with our proof technique, the Hadamard conjecture only increases the lower bound to $\frac{1}{4}$.
Meanwhile, our explicit construction involving symmetric conference matrices suggests taking $\alpha$ to be $\frac{1}{2}$.

Next, we are interested in the optimal version of Dong and Rudelson's original result in~\cite{DongR:24}:

\begin{problem}
What is $\sup_n\kappa(n)$?
\end{problem}

Of course, our result $\kappa(n)\to1$ reduces this to a finite problem.
Sadly, it's not all that finite:
even assuming the Hadamard conjecture, a refinement of our upper bound on $\kappa(n)$ establishes that $\kappa(n)<2$ for every $n > 3.11 \times 10^6$.
For the record, the authors believe $\sup_n\kappa(n)=2$, and to prove this, it suffices to treat odd $n$, since taking a Kronecker product with a $2\times 2$ Hadamard matrix would cover the even case.
For larger values of $n$ that are not resolved by the asymptotic result, it might suffice to find circulant approximate Hadamard matrices, considering Steinerberger's main result in~\cite{Steinerberger:24}.

Unfortunately, we didn't find many explicit examples to help close this gap.
At the moment, we observe a ``$\sqrt{2}$ bottleneck'' in the case where $n$ is odd:

\begin{problem}
Is there an explicit infinite family of approximate Hadamard matrices of odd order with condition number less than $\sqrt{2}$?
\end{problem}


\section*{Acknowledgments}

The authors thank the anonymous referee for suggestions that improved the presentation of our results.
DGM thanks Stefan Steinerberger for introducing him to this problem.
BA and DGM thank OpenAI for access to GPT-5 Pro, and Harmonic for access to Aristotle.
The authors used ChatGPT to write exploratory code and search the literature for this project. 
ChatGPT also surfaced Lemma~\ref{lem.submatrix orthogonalization} as a useful step in the proof of Theorem~\ref{thm.flat orthogonal matrices}, and out of an abundance of caution, the authors used Aristotle to formalize its proof in Lean.
See the appendix of the arXiv version of this paper for more details.
All mathematical claims  were checked by the authors, who take full responsibility for the content of this paper.
JJ was supported by NSF DMS 2220320. 
DGM was supported by NSF DMS 2220304. 
The views expressed are those of the authors and do not reflect the official guidance or position of the United States Government, the Department of Defense, the United States Air Force, or the United States Space Force.

\appendix

\section{Table}

Table~\ref{table} reports conjectural values of $\kappa(n)$ for $n\leq 30$ (in cases other than $1$, $2$, and multiples of $4$).
In addition to the first $10$ digits of the condition number, we supply its minimal polynomial.
By exhaustive search, the reported condition numbers are known to equal $\kappa(n)$ for $n\leq 6$.
Interestingly, for smaller values of $n$, matrices that appear to minimize condition number also maximize the determinant (in absolute value).
Also, when $n$ is even, we found a lot of success with circulant $n/2\times n/2$ blocks, much like the SDS construction.
(Each matrix appears as an ancillary file of the arXiv version of this paper.)

\begin{table}
\centering
\caption{Smallest known condition numbers}
\label{table}
\bigskip

\setlength{\aboverulesep}{0pt}
\setlength{\belowrulesep}{0pt}

\setlength{\extrarowheight}{3pt}

\begin{tabular}{r c l l}
\toprule
\rowcolor{toprow}
$n$ & condition number & minimal polynomial & comments \\[3pt]
\midrule

3  & 2.000000000 & $-2+t$ & 
circulant, symmetric, \\[-3pt]
   &             &         & max determinant \\[3pt]
\midrule

5  & 1.500000000 & $-3+2t$ & Barba, circulant, symmetric, \\[-3pt]
   &             &         & max determinant \\[3pt]
\midrule

\rowcolor{evenrow}
6  & 1.581138830 & $-5+2t^{2}$ & SDS, circulant $3\times 3$ blocks, \\
\rowcolor{evenrow}
   &             &              & symmetric, max determinant \\[3pt]
\midrule

7  & 1.732050808 & $-3+t^{2}$ & symmetric, max determinant \\[3pt]
\midrule

9  & 1.850781059 & $-5 - t + 2t^{2}$ & symmetric \\[3pt]
\midrule

\rowcolor{evenrow}
10 & 1.500000000 & $-3+2t$ & SDS, circulant $5\times5$ blocks, \\[-3pt]
\rowcolor{evenrow}
   &             &         & max determinant \\[3pt]
\midrule

11 & 1.767766953 & $-25+8t^{2}$ & symmetric, max determinant \\[3pt]
\midrule

13 & 1.443375673 & $-25+12t^{2}$ & Barba, symmetric, \\[-3pt]
   &             &               & max determinant \\[3pt]
\midrule

\rowcolor{evenrow}
14 & 1.471960144 & $-13+6t^{2}$ & SDS, circulant $7\times7$ blocks, \\[-3pt]
\rowcolor{evenrow}
   &             &              & max determinant \\[3pt]
\midrule

15 & 1.527525232 & $-7+3t^{2}$ & symmetric, max determinant \\[3pt]
\midrule

17 & 1.700930833 &
\footnotesize{$256 - 1152 t^{2} + 1661 t^{4} - 936 t^{6} + 169 t^{8}$}
& symmetric \\[3pt]
\midrule

\rowcolor{evenrow}
18 & 1.457737974 & $-17+8t^{2}$ & SDS, circulant $9\times9$ blocks, \\[-3pt]
\rowcolor{evenrow}
   &             &             & max determinant \\[3pt]
\midrule

19 & 1.662877383 &
$-77 + 288 t^{2} - 307 t^{4} + 77 t^{6}$ & circulant \\[3pt]
\midrule

21 & 1.732050808 & $-3+t^{2}$ & circulant core \\[3pt]
\midrule

\rowcolor{evenrow}
22 & 1.511424872 &
\footnotesize{$\begin{aligned}
-3719 + 21191 t^{2} - 45561 t^{4} + 45825 t^{6}\\[-3pt]
 -21466 t^{8} + 3719 t^{10}
\end{aligned}$}
& circulant $11\times11$ blocks \\[3pt]
\midrule

23 & 1.702109681 &
\footnotesize{$\begin{aligned}
-6029 + 40001 t^{2} - 93367 t^{4} + 93950 t^{6}\\[-3pt]
 -40331 t^{8} + 6029 t^{10}
\end{aligned}$}
& circulant core \\[3pt]
\midrule

25 & 1.428869017 & $-49 + 24t^{2}$ & Barba, max determinant \\[3pt]
\midrule

\rowcolor{evenrow}
26 & 1.329508134 & $3 - 7t^{2} + 3t^{4}$ & circulant $13\times13$ blocks \\[3pt]
\midrule

27 & 1.603484352 &
$-271 + 1029 t^{2} - 1056 t^{4} + 271 t^{6}$
& block circulant with \\[-3pt]
& & & circulant $9\times9$ blocks \\[3pt]
\midrule

29 & 1.666939342 &
\footnotesize{$\begin{aligned}
354061 - 1045624 t^{2} + 1266560 t^{4} - 806784 t^{6}\\[-3pt]
 + 285440 t^{8} - 53248 t^{10} + 4096 t^{12}
\end{aligned}$}
& circulant core \\[3pt]
\midrule

\rowcolor{evenrow}
30 & 1.379101101 & $59 - 109t^{2} + 41t^{4}$ & circulant $15\times15$ blocks \\[3pt]
\bottomrule
\end{tabular}
\end{table}

\section{Comments on AI}
\label{sec.ai}

\textit{N.B.: The content of this section is most naturally presented as a first-person account, so it is written from the perspective of the third author.}

\medskip

\noindent
For the first month, this project did not involve any AI. 
We had two main ideas:
\begin{itemize}
\item[1.]
Inspired by the known circulant approximate Hadamard matrices obtained by randomly perturbing the Legendre symbol, we figured out a non-circulant analogue:
Given a Hadamard matrix of order $n+1$, normalize the first row and column before discarding them, and then randomly flip a few $-1$'s to $+1$'s to get condition number $1+c\cdot n^{-1/4}\sqrt{\log n}$.
\item[2.]
Inspired by how John and I have studied projective codes~\cite{JasperKM:19}, we ran a bunch of computer experiments to minimize the condition number for $n \leq 30$, with the intent of hunting for emergent combinatorial structure that would then lead to infinite families.
\end{itemize}

We wanted to extend our first idea in a way that starts with an approximate Hadamard matrix, thereby giving an inductive construction of approximate Hadamard matrices of order $m-k$ whenever there's a Hadamard matrix of order $m$. 
We expected to be able to get $k$ up to $\sqrt{m}$ or so since each random perturbation would increase the fraction of $+1$s by about $1/\sqrt{m}$. 
We also knew from ChatGPT (the free version) that Hadamard gaps are currently known to be much smaller than $\sqrt{m}$, so this would be a route to prove $\kappa(n)=1+o(1)$. 
However, we didn't have a clear picture of how the induction would work, so we hesitated to write the details.

Recently, I had upgraded to GPT-5 Pro.
I was throwing big open problems at it, but the problems were too hard (or I was bad at prompting), and frustratingly, Pro tended to push back on thinking about famous open problems.
(Maybe it was trained to save on compute?) 
Finally, it occurred to me that I should have Pro write exploratory code in support of our second main idea. 
After all, this is an intended use case of Pro.

\subsection{Exploratory code}

About a month prior, we launched a private table of best known approximate Hadamard matrices. 
My only surviving contributions were for $n\in\{3,5,6,9\}$.
But then I ran some code written by Pro, and it improved the $n=13$ case from $1.7677$ to $1.4433$. (!) 
Notably, I didn't tell Pro how it should perform the optimization (and I never bothered to find out what it did); it just selected an approach and then outperformed about a month's worth of expert searching. 
Sadly, the code didn't improve any other entries, so perhaps this was just a weak point in the table.

The drop from $1.7677$ to $1.4433$ was so big that we expected some hidden combinatorial structure in the matrix. 
So I asked Pro for an explanation involving combinatorial design, and it provided a plausible explanation in terms of the projective plane $\operatorname{PG}(2,3)$.
At this point, I was hoping for an infinite family, so I asked Pro, but other constructions involving projective planes have larger condition numbers; this was a fluke.
After more back and forth, Pro realized this matrix achieves equality in the \textit{Barba bound} for Hadamard's maximal determinant problem. 
By consulting a recent survey on this problem~\cite{BrowneEHC:21} (also pointed to by Pro), we identified the infinite family in Section~\ref{sec.explicit families} with condition number $\sqrt{2}+o(1)$ and improved the $n=25$ entry in our table.

(I wouldn't have expected any relationship between Hadamard's maximal determinant problem and approximate Hadamard matrices, but in hindsight, there's a nice analogy with energy minimization on the sphere.
In particular, there are special configurations of points that simultaneously minimize a diverse family of energies~\cite{CohnK:07}.
So while I can't prove a strong relationship between the determinant and the condition number, it still makes sense to evaluate the condition number of determinant maximizers on the off chance they're sufficiently universally optimal.)

\subsection{Proof discovery}

Part of me was hoping the exploratory code would inspire enough infinite families to render our first main idea moot.
But after exhausting that avenue, it was time to return to the induction.
Since it wasn't clear to us how the induction should go, I expected it to take a full day to work out one approach, only to see how it breaks. 
Maybe Pro could help navigate this for me?

I started by telling Pro a sketch of the result we already had, and I asked it to flesh out the details. 
I actually forgot to mention that we used matrix Bernstein, but it figured it out. 
A perfect reconstruction of the proof on the first try. 
So then I explained how we wanted to extend this result to get even smaller approximate Hadamard matrices, and I told it my vague concepts of a plan of how an induction might go. 
And it responded with a proof! 
But the proof was wrong. 
After lots of back and forth, it started to push back on the inductive approach entirely.
Instead, it wanted to push towards a ``one-shot'' approach, but I wasn't convinced. 

I asked how the one-shot approach would treat the $k=1$ case we already understood, and it described an expected matrix that was a rank-$1$ perturbation (involving the discarded row and column vectors) of the surviving $n\times n$ matrix (before normalization). 
But in general, we want this expected matrix to be a multiple of an orthogonal matrix, and that seemed difficult to me, so I asked a more general question (paraphrased):
\begin{quote}
Given an orthogonal matrix with block decomposition $\left[\begin{smallmatrix} A & B \\ C & D \end{smallmatrix}\right]$, how can you use $A$, $B$, $C$, and $D$ to construct an orthogonal matrix of the same size as $D$? 
\end{quote}
After some thought, it came up with 
\[
D+C(I-A)^{-1}B.
\]
I didn't believe it. 
I asked for a proof, and it gave a full page of calculations that I wasn't interested in checking. 
So I punched the formula into MATLAB with a random orthogonal matrix, and the result was orthogonal!

Where did this formula come from? 
Pro points to references from linear system theory, scattering theory, and operator theory (e.g.,~\cite{Alpay:02}), but I would never have thought to consult these corners of the literature.
Case in point, this exact formula for submatrix orthogonalization is the subject of a reference request on MathOverflow~\cite{Terrell:mo}, but no one offers a pointer to the existing literature.

Once I had this submatrix orthogonalization formula (see Lemma~\ref{lem.submatrix orthogonalization}), I was able to sketch a ``one-shot'' proof of the desired result to Pro, and it responded with all the details fleshed out (much like my experience with the original $k=1$ case). 
It's clear that Pro is much better at writing a proof from a sketch than from scratch. 
(Aren't we all?) 
But here's the most humbling thing about this experience: 
I don't know if I could've proven this result without the submatrix orthogonalization formula, and I don't think I would've come up with that formula on my own.

(More generally, how many open problems could be solved today if the experts knew the relevant results from other communities?
I'm optimistic that AI will help to close this gap.)

\subsection{Formalization}

The only theoretical component of this paper that came from AI is submatrix orthogonalization (Lemma~\ref{lem.submatrix orthogonalization}).
Out of an abundance of caution, we decided to formalize the proof of this result in Lean~\cite{MouraU:21}.
Instead of vibe coding a proof with ChatGPT as in~\cite{AlexeevM:25}, we adopted a much more user-friendly auto-formalization workflow.
Boris asked Pro to state the result in Lean, with a couple pointers to block matrix concepts like \texttt{fromBlocks}. 
After 9 minutes of thinking, it responded with a formal statement that looked very clean and clearly corresponded exactly to the intended result.
Boris then asked Aristotle~\cite{AchimEtal:25} to prove the result, offering no hints whatsoever. 
After 13 minutes, it responded with a formal proof, and it type-checked with no issues.
(The Lean code appears as an ancillary file of the arXiv version of this paper.)

At this point, we've verified Lemma~\ref{lem.submatrix orthogonalization} in four independent ways: (1) we checked it in MATLAB with random orthogonal matrices, (2) we identified references in the literature that state the result, (3) we wrote out a human-readable proof, and (4) we auto-formalized a proof in Lean.
The result is legit. 
I have no intuition for it, but I know it's true.

\end{document}